\documentclass[a4paper,12pt]{article}
\usepackage[utf8]{inputenc}
\title{Resolvent estimates and numerical implementation for the homogenisation of one-dimensional periodic mixed type problems}
\author{Sebastian Franz\footnote{
          Institute of Scientific Computing, Technische Universit\"at Dresden,
          01062 Dresden, Germany.
          \mbox{e-mail}: sebastian.franz@tu-dresden.de}\ \ and 
        Marcus Waurick\footnote{
          Dept. of Mathematics and Statistics,
          University of Strathclyde,
          Glasgow, UK.
          \mbox{e-mail}: marcus.waurick@strath.ac.uk}
        }
\date{\today}

\usepackage{fancyhdr} 
\usepackage{pgfpages}
\usepackage{amsmath}
\usepackage{amssymb}
\usepackage{amsthm}

\makeatletter
\let\my@saved@original@eqref\eqref 
\renewcommand*{\eqref}[1]{
  \begingroup
    \let\normalfont\relax
    \my@saved@original@eqref{#1}
  \endgroup
}
\makeatother

\usepackage{cite}

\usepackage{color}

\newcommand{\e}{\mathrm{e}}

\newcommand{\jump}[1]{[\hspace*{-2pt}[#1]\hspace*{-2pt}]}

\newcommand{\eps}{\varepsilon}

\newcommand{\scp}[1]{\langle #1 \rangle_H}

\newcommand{\N}{\mathbb{N}}

\newcommand{\R}{\mathbb{R}}

\newcommand{\U}{\mathcal{U}}
\newcommand{\PS}{\mathcal{P}}

\newcounter{tmp}
\newcommand{\makeballnumber}[1]{\setcounter{tmp}{\theenumi}%
\setcounter{enumi}{#1}%
\leavevmode \csname beamer@@tmpl@enumerate item\endcsname%
\setcounter{enumi}{\thetmp}}
\newcommand{\makeball}{\leavevmode \csname beamer@@tmpl@itemize item\endcsname}

\renewcommand{\Re}{\mbox{Re}}

\definecolor{seb}{rgb}{0.9,0,0}


\usepackage{hyperref}
\usepackage{mathtools}

\usepackage{tikz}
\usepackage{pgfplots}
\pgfplotsset{compat=1.9}
\usetikzlibrary{calc}
\usetikzlibrary{matrix}
\usetikzlibrary{external}
\tikzsetfigurename{figure_}
\tikzset{external/system call={pdflatex \tikzexternalcheckshellescape -interaction=batchmode -jobname "\image" "\texsource";
convert -density 600 -transparent white "\image.pdf" "\image.png"}}

\renewcommand{\phi}{\varphi}
\DeclareMathAlphabet{\mathcal}{OMS}{cmsy}{m}{n}

\theoremstyle{definition}
\newtheorem*{definition}{Definition}

\theoremstyle{plain}
\newtheorem{thm}{Theorem}[section]

\newtheorem{proposition}[thm]{Proposition}
\newtheorem{cor}[thm]{Corollary}
\newtheorem{rem}[thm]{Remark}

\usepackage{bbm}

\newcommand{\scprm}[1]{\langle #1 \rangle_{\rho,m}}

\newcommand{\tmi}{t_{m,i}}

\newcommand{\Qmr}[1]{Q_m\left[#1\right]_\rho}

\DeclareMathOperator{\rge}{ran}

\DeclareMathOperator{\diag}{diag}
\begin{document}
 \pagestyle{fancy}
  \maketitle
  \begin{abstract}
    We study a homogenisation problem for problems of mixed type in the 
    framework of evolutionary equations. The change of type is highly oscillatory. 
    The numerical treatment is done by a discontinuous Galerkin method 
    in time and a continuous Galerkin method in space.
  \end{abstract}



 \section{Introduction}
  A standard problem in engineering is the approximation of highly oscillatory coefficients by averaged ones. In fact, given a partial differential equation with variable coefficients, numerical procedures might be too involved for nowadays computing devices so that an effective model is often derived. The process of seeking effective coefficients as replacements for highly oscillatory ones is summarised under the umbrella term of homogenisation. The mathematical theory of homogenisation goes back to the late 1960s. We refer to the standard references \cite{Bensoussan1978,Cioranescu1999} for a more detailed account. 
  
  Standard applications of homogenisation are elliptic, parabolic or hyperbolic divergence form equations. Only quite recently, \cite{W16_SH} it has been noticed that for certain problems of mixed type, that is, differential equations changing their type from hyperbolic to parabolic to elliptic on different spatial domains in a highly oscillatory way, one can derive an effective model, which does not change its type anymore and consists of constant coefficients. 
  
  In \cite{W16_SH}, only a qualitative convergence statement was derived. The techniques developed in \cite{CW17_1D,CW17_FH}, however, suggest that the rate of convergence can be quantified. It is one main result of the present exposition -- based on the rationale outlined in \cite{CW17_1D,CW17_FH} -- that a quantified convergence rate for problems of the type discussed in \cite{W16_SH} can be derived. We refer to Section \ref{sec:analysis} for the precise equations.
  
  Given the low dimensionality of the problem to be discussed in this paper, we will furthermore numerically study the partial differential equation with highly oscillatory coefficients and provide a quantitative convergence statement that for highly oscillatory coefficients the corresponding numerical solution approximates the true solution of the homogenised model. In fact, the results in \cite{FrTW16} show that for mixed type equations one can derive a numerical scheme. It consists of a discontinuous Galerkin method in time, see e.g. \cite{RH73,Cockburn2000,Riviere08},  combined with a continuous Galerkin method in space. The framework developed in \cite{FrTW16} for a slightly different setting can be extended to our present problem easily and approximation properties proved therein can be transferred. 
  
  In Section \ref{sec:analysis}, we introduce the model under consideration and provide the desired convergence statement. In Section \ref{sec:numerics} we recall the numerical scheme derived in \cite{FrTW16} and provide the estimate that the numerical solution of the equation with highly oscillatory coefficients approximates the solution of the effective equation in a certain controlled way. We conclude this paper with a short numerical example in Section \ref{sec:simulation}.
     
  \section{Resolvent estimates for the continuous in-time homogenisation problem}\label{sec:analysis}
  
  In \cite{FrTW16}, we have already established the well-posedness of the Galerkin approximations and convergence to the original problem. What we aim to establish here is in spirit similar to the approach developed in \cite{CW17_1D,CW17_FH}. The main ingredients for this one-dimensional situation can readily be found in \cite{CW17_1D}. The main difference between the cases treated in \cite{CW17_1D} or \cite{CW17_FH} is the underlying spatial domain. In fact, the cited work focused on $\mathbb{R}$ and $\mathbb{R}^d$ as underlying spatial domain. In the present case, we treat the unit interval, instead. More precisely, using the formulation in \cite{W16_SH}, one can write the problem in question as the following $2\times 2$-block operator matrix system:
  \begin{equation}\label{eq:Lu}
      \left(\partial_t M_0(N\cdot) + M_1(N\cdot)+\begin{pmatrix}0&\partial_\# \\ \partial_\# &0\end{pmatrix}\right) U_N= F,
  \end{equation}
  where $\partial_\#$ is the weak derivative on $(0,1)$ with periodic boundary conditions, $M_0,M_1$ are $1$-periodic, measurable bounded $\mathbb{C}^{2\times 2}$-valued functions with the additional property that $M_0(x)=M_0(x)^\ast\geq 0$ and that there exists $\rho>0$ and $c>0$ such that
  \[
      \rho \langle M_0(x)\xi,\xi\rangle_{\mathbb{C}^2}+ \Re  \langle M_1(x)\xi,\xi\rangle_{\mathbb{C}^2}\geq c\langle \xi,\xi\rangle_{\mathbb{C}^2}.
  \]    
  As the latter equation is formulated on $(0,1)$, the continuous Gelfand transformation used in \cite{CW17_1D} to divide the problem on the whole space has to replaced by its discrete analogue. In the next two subsections, we will derive an estimate for the static case, which will eventually be applied to the dynamic case by going into the frequency domain. 
  \subsection{The static case}
    We start out with the discrete analogue of the Gelfand transformation as introduced in \cite{CW17_FH}.
    \begin{definition} Let $N\in\mathbb{N}$, $f\colon \mathbb{R}\to\mathbb{C}$. Then we define
    \[
       \mathcal{V}_N f(\theta,y)\coloneqq \frac{1}{\sqrt{N}}\sum_{k=0}^{N-1}f(y+k)e^{-i\theta k} \quad (y\in [0,1),\theta\in \{2\pi k/N; k\in\{0,\ldots, N-1\}).
    \]     
    \end{definition}
    \begin{proposition}
      The operator $\mathcal{V}_N\colon L^2_\#(0,N)\to L^2(0,1)^N$ given by
      \[
          f\mapsto \left(\mathcal{V}_N f(2\pi (k-1)/N,\cdot)\right)_{k\in\{1,\ldots,N\}}
      \]      
      is unitary, where $L^2_\#(0,N)\coloneqq \{f\in L^2_\textnormal{loc}(\mathbb{R}); f(\cdot+Nk)=f\quad (k\in\mathbb{Z})\}$ endowed with the norm of $L^2(0,N)$.
    \end{proposition}
    \begin{proof}
     Let $f\colon \mathbb{R}\to \mathbb{C}$ be bounded, continuous with $ f(\cdot+Nk)=f$ for all $k\in\mathbb{Z}$. Then, we compute with $\theta_\ell=2\pi \ell/N$
     \begin{align*}
        N\|\mathcal{V}_N f\|^2_{L^2(0,1)^N} &= \sum_{\ell=0}^{N-1} \|\mathcal{V}_N f(2\pi \ell/N,\cdot)\|^2_{L^2(0,1)}
        \\ & = \sum_{\ell=0}^{N-1}\sum_{k_1=0}^{N-1}\sum_{k_2=0}^{N-1} e^{-i\theta_{\ell}(k_1-k_2)}\int_{(0,1)} f(y+k_1)\overline{f(y+k_2)} dy.
     \end{align*}
     We shall argue next that for all $k_1,k_2\in\{0,\ldots,N-1\}$ with $k_1\neq k_2$, we have
     \begin{equation}\label{eq:1a}
         \sum_{\ell=0}^{N-1} e^{-i\theta_{\ell}(k_1-k_2)} = 0. 
     \end{equation}
     For this, denote $n\coloneqq k_1-k_2\neq 0$ and consider the homomorphism
     \begin{align*}
        \phi \colon \mathbb{Z}_{N} &\to G\coloneqq \{ e^{-i\frac{2\pi n}{N}\ell} ; \ell\in\{0,\ldots,N-1\}\}
        \\                        \ell&\mapsto e^{-i\frac{2\pi n}{N}\ell}.
     \end{align*}     
     By the fundamental theorem on homomorphisms, $G=\rge(\phi)\cong \mathbb{Z}_N/\ker(\phi)$. In particular, $|G|$ divides $N$. Furthermore, since $\mathbb{Z}_N$ is cyclic, we obtain that $\mathbb{Z}_N/\ker(\phi)$ is cyclic and thus $G$ is cyclic. Let $z_\ast\in G$ generate $G$. Thus, $G=\{z_\ast^{0},\ldots,z_\ast^{k-1}\}$ are the $k$ unique, distinct $k$th unit roots. In particular, we obtain for all $z\in\mathbb{C}$
     \[
        z^k-1 = (z-z^{0}_\ast)\cdot \ldots (z-z^{k-1}_\ast).
     \]
     Expanding the right-hand side and comparing the coefficient of $z^{k-1}$ of both sides, we deduce that 
     \[
         \sum_{z\in G} z = 0.
     \]
     Hence, 
     \[
         \sum_{\ell=0}^{N-1} e^{-i\theta_{\ell}n}=\sum_{\ell=0}^{N-1} \phi(\ell)= \frac{N}{|G|} \sum_{z\in G} z = 0,
     \]
     which settles \eqref{eq:1a}. Therefore, we obtain
     \[
        N\|\mathcal{V}_N f\|^2_{L^2(0,1)^N}=\sum_{\ell=0}^{N-1}\sum_{k_1=0}^{N-1} \int_{(0,1)} f(y+k_1)\overline{f(y+k_1)} dy = N \|f\|_{L^2(0,N)}^2.
     \]     
     Moreover, note that for $\phi\in C_c(0,1)^N$, we have that the $N$-periodic extension of $f$ given by $f(x)=e^{i2\pi k/N} \phi_{k+1}(x)$ for $x\in [k,k+1)$ with $k\in\{0,\ldots,N-1\}$ leads to $N\mathcal{V}_N f=\phi$. Hence, $\mathcal{V}_N$ has dense range. Thus, $\mathcal{V}_N$ is unitary.    
    \end{proof}
    We shall furthermore introduce the following unitary scaling transformation that scales a problem on $(0,1)$ onto $(0,N)$:
    \begin{definition}
      Let $N\in\mathbb{N}$. Then define for $f\in L^2(0,1)$
      \[
       \mathcal{T}_N f \coloneqq {\tfrac{1}{\sqrt{N}}} f\big(\tfrac{\cdot}{N}\big)
      \]
      and $\mathcal{G}_N\coloneqq \mathcal{V}_N\mathcal{T}_N$
    \end{definition}
    Moreover, we define
    \[
       \partial_\theta \colon H^1_\theta(0,1)\subseteq L^2_\#(0,1)\to L^2_\#(0,1), f\mapsto f'
    \]
    and $H^1_\theta(0,1)=\{f\in H^1(0,1); f(1)=e^{i\theta}f(0)\}$. We use $\partial_\#$ and $H^1_\#(0,1)$, if $\theta=0$.
    \begin{proposition}\label{prop:GT1} Let $N\in\mathbb{N}$. Then
    \begin{enumerate}
     \item[(a)] $\mathcal{T}_N\partial_{\#} =N \partial_{\#,N} \mathcal{T}_N$, where $\partial_{\#,N}$ is the weak derivative with periodic boundary conditions,
     \item[(b)] $\mathcal{G}_N\partial_\# = N\diag\big((\partial_{\theta_k})_{k\in\{0,\ldots,N-1\}}\big)\mathcal{G}_N$, where $\theta_k= 2\pi k/N$.
     \item[(c)] For all $a\in L^\infty_\#(0,1)$ we obtain $\mathcal{G}_N a(N\cdot) = \diag \big((a(\cdot))_{k\in\{0,\ldots,N-1\}}\big)\mathcal{G}_N$.
    \end{enumerate}     
    \end{proposition}
    \begin{proof}
     The proof follows along elementary calculations. Note that for (a) and (b) it suffices to prove the assertions for smooth functions, only.
    \end{proof}
    Next, we introduce a static version of the problem in question:
    \begin{definition}
      Let $c>0$ and 
      \[
        \mathcal{M}_c \coloneqq \{ M\in L^\infty(0,1)_\#^{2\times 2}; \Re M\geq c1_{2\times 2}\},
      \]
      where $L^\infty(0,1)_\#\coloneqq\{ a\in L^\infty(\mathbb{R}); a(\cdot+k)=a\quad (k\in\mathbb{Z})\}$.
    \end{definition}
    For all $N\in\mathbb{N}$, find $\begin{pmatrix}u_N \\ v_N\end{pmatrix}\in L^2(0,1)^2$ such that
    \begin{equation}\label{eq:Npropstat}
    \left(  M(N\cdot)+\begin{pmatrix}
                 0 & \partial_\# \\ \partial_\# & 0
                \end{pmatrix}\right)\begin{pmatrix}u_N \\ v_N\end{pmatrix} = \begin{pmatrix}f \\ g\end{pmatrix}
    \end{equation}
  for some $f,g\in L^2(0,1)^2$. Note that \eqref{eq:Npropstat} is well-posedness by \cite[Lemma 2.5]{CW17_FH}. With the help of Proposition \ref{prop:GT1}, we obtain an equivalent formulation of \eqref{eq:Npropstat}
    \begin{cor}\label{cor:GT2} Let $N\in \mathbb{N}$. Then 
    \begin{multline*}
       \begin{pmatrix}
                 \mathcal{G}_N & 0 \\ 0 & \mathcal{G}_N
                \end{pmatrix} \left(  M(N\cdot)+\begin{pmatrix}
                 0 & \partial_\# \\ \partial_\# & 0
                \end{pmatrix}\right) \begin{pmatrix}
                 \mathcal{G}_N & 0 \\ 0 & \mathcal{G}_N
                \end{pmatrix}^*\\ = \left( \diag( M(\cdot))_{k\in\{0,\ldots,N-1\}}+{N}\diag\left(\begin{pmatrix}
                  0 & \partial_{\theta_k} \\ \partial_{\theta_k} & 0
                 \end{pmatrix}
\right)_{k\in\{0,\ldots,N-1\}}\right).
    \end{multline*}
    \end{cor}
    As it has been demonstrated in \cite[Section 3]{CW17_FH}, we obtain that \cite[Theorem 2.4 and Theorem 2.2]{CW17_FH} applies to the setting in \cite[Equation (10)]{CW17_FH}. Here we recall the results found there for the particular case of $n=d=1$. Note that by \cite[Remark 4.6]{CW17_FH} the one-dimensional homogenised coefficient is given by the integral mean.
    \begin{thm}\label{thm:homstat} For all $N\in\mathbb{N}$ and $k\in\{0,\ldots,N-1\}$, we have
    \begin{multline*}
     \left\| M(\cdot)+{N}\begin{pmatrix}
                  0 & \partial_{\theta_k} \\ \partial_{\theta_k} & 0
                 \end{pmatrix} - \left(\int_{(0,1)}M(y)dy+{N}\begin{pmatrix}
                  0 & \partial_{\theta_k} \\ \partial_{\theta_k} & 0
                 \end{pmatrix}\right)\right\| \\ \leq \frac{1}{\pi}\left(2\left(1+\frac{\|M\|_\infty}{c}\right)^2+1\right)\frac{1}{N}.
    \end{multline*}     
    \end{thm}
    \subsection{The dynamic case}
    With the estimate in the latter theorem, we obtain also result for the full time-dependent problem. The strategy has been outlined in the concluding sections of \cite{CW17_1D} already. We will, however, provide the necessary notions and a corresponding estimate in this exposition, as well. For $\rho>0$ and a Hilbert space $H$, we define 
    \[
       L_\rho^2(\mathbb{R};H)\coloneqq \{ f\colon\mathbb{R}\to H; f \text{ measurable}, \int_{\mathbb{R}}\|f(t)\|_H^2\exp(-2\rho t)dt<\infty\},
    \]endowed with the obvious scalar product. Employing the usual identification of functions being equal almost everywhere, we obtain that $L_\rho^2(\mathbb{R};H)$ is a Hilbert space. We denote by $H_\rho^1(\mathbb{R};H)$ the first Sobolev space of weakly differentiable functions with weak derivative being representable as an element of $L_\rho^2(\mathbb{R};H)$. Then we put
    \[
      \partial_t \colon H^1_\rho(\mathbb{R};H)\subseteq L_\rho^2(\mathbb{R};H)\to L_\rho^2(\mathbb{R};H), f\mapsto f'.
    \]
    A spectral representation of $\partial_t$ as multiplication operator is given by the \emph{Fourier--Laplace transformation}, that is, the unitary extension of the operator $\mathcal{L}_\rho\colon L_\rho^2(\mathbb{R}:H)\to L^2(\mathbb{R}:H)$ given by
    \[
      \mathcal{L}_\rho \phi(\xi)=\frac{1}{\sqrt{2\pi}}\int_\mathbb{R} \phi(t)\exp(-it\xi-\rho t)dt\quad(\phi\in C_c(\mathbb{R};H)),
    \]
    where $C_c(\mathbb{R};H)$ is the space of continuous functions with compact support. The spectral representation reads as follows:
    \begin{thm}[{\cite[Corollary 2.5]{KPSTW14_OD}}]
\label{thm:FLT}Let $\rho\in\mathbb{R}$. Then
\[
\partial_{t}=\mathcal{L}_{\rho}^{\ast}(im+\rho)\mathcal{L}_{\rho},
\]where
\begin{align*}
m\colon\{f\in L_{2}(\mathbb{R};H);(t\mapsto tf(t))\in L_{2}(\mathbb{R};H)\}\subseteq L_{2}(\mathbb{R};H) & \to L_{2}(\mathbb{R};H)\\
f & \mapsto(t\mapsto tf(t))
\end{align*}
is the multiplication by the argument operator with maximal domain.
\end{thm}
Next, we recall an elementary version of the well-posedness theorem for evolutionary equations, which is particularly relevant to the case studied here. For this, note that we will use the same notation for an operator acting in $H$ and its corresponding lift as an abstract multiplication 
operator on $L_\rho^2(\mathbb{R};H)$.
\begin{thm}[{{\cite[Solution Theory]{PicPhy}, \cite[Theorem 6.2.5]{Picard}}}]\label{t:st} Let $A$ be a skew-selfadjoint operator in ${H}$, $0\leq M_0=M_0^*,M_1\in L({H})$. Assume there exists $c,\rho>0$ with
\begin{equation}\label{e:st}
   \rho\langle  M_0 \phi,\phi\rangle + \Re \langle M_1 \phi,\phi \rangle \geq c\langle \phi,\phi\rangle\quad  (\phi\in \mathcal{H}).
\end{equation}
Then the operator $\mathcal{B} \coloneqq \partial_{t} {M}_0+{M}_1+{A}$ with $D(\mathcal{B})=D(\partial_{t})\cap D(\mathcal{A})$ is closable in $L_{\rho}^2(\mathbb{R};{H})$. Moreover, $\mathcal{S}_\rho \coloneqq \overline{\mathcal{B}}^{-1}$ is well-defined, continuous and bounded with $\|\mathcal{S}_\rho\|_{L(L_\rho^2)}\leq 1/c$.  
\end{thm}

We can now state and prove the full time-dependent version of Theorem \ref{thm:homstat}. We shall also refer to \cite[Theorem 7.1]{CW17_1D} for a corresponding result with $\mathbb{R}$ instead of $(0,1)$ as underlying state space.

\begin{thm}\label{thm:conv_analy} Let $\rho>0$, $M_0,M_1\in L^\infty(0,1)_\#^{2\times 2} (\subseteq L(L^2(0,1)^2))$, $M_0=M_0^\ast\geq 0$. Assume there exists $c>0$ such that
\[
   \rho \langle M_0\phi,\phi\rangle+\Re\langle M_1\phi,\phi\rangle \geq c\langle \phi,\phi\rangle \quad (\phi\in L^2(0,1)^2),
\]
set $A\coloneqq \left(\begin{smallmatrix} 0 &\partial_\# \\ \partial_\# & 0\end{smallmatrix}\right)$, $H=L^2(0,1)^2$.
Then, there exists $\kappa\geq 0$ such that for all $N\in \mathbb{N}$, we have
\[
    \| \left((\partial_t M_0(N\cdot)+M_1(N\cdot) + A)^{-1}-(\partial_t M_0^{\textnormal{av}}+M_1^{\textnormal{av}} + A)^{-1}\right)\partial_t^{-2}\|_{L(L_\rho^2(\mathbb{R};H))}\leq \frac{\kappa}{N},
\] 
where $M_j^\textnormal{av}\coloneqq \int_{(0,1)}M_j(y)dy$ for all $j\in\{0,1\}$.
\end{thm}
\begin{proof}
 Applying the unitarity of the Fourier--Laplace transformation and the spectral representation of $\partial_t$, we deduce that the claim is equivalent to showing that there exists $\kappa\geq 0$ such that for all $N\in\mathbb{N}$ and $\xi\in\mathbb{R}$:
 \begin{equation}\label{eq:FLTest}
      \| \left(((i\xi+\rho) M_0(N\cdot)+M_1(N\cdot) + A)^{-1}-((i\xi+\rho) M_0^{\textnormal{av}}+M_1^{\textnormal{av}} + A)^{-1}\right)(i\xi+\rho)^{-2}\|_{L(H)}\leq \frac{\kappa}{N}.
 \end{equation}For this, we deduce from the positive definiteness estimate imposed on $M_0$ and $M_1$ that
 \[
    (i\xi+\rho) M_0(\cdot)+M_1(\cdot) \in \mathcal{M}_c
 \]
 for all $\xi \in \mathbb{R}$. Hence, using Theorem \ref{thm:homstat} and Corollary \ref{cor:GT2}, we obtain the existence of $\kappa\geq 0$ such that for all $N\in\mathbb{N}$ and $\xi\in\mathbb{R}$
 \begin{multline*}
       \| \left(((i\xi+\rho) M_0(N\cdot)+M_1(N\cdot) + A)^{-1}-((i\xi+\rho) M_0^{\textnormal{av}}+M_1^{\textnormal{av}} + A)^{-1}\right)\|_{L(H))} \\ \leq  \frac{\kappa}{N}(1+|\xi|^2)(1+\|M_0\|_\infty+\|M_1\|_\infty)^2.
 \end{multline*}
 Thus, we conclude
 \begin{align*}
       & \| \left(((i\xi+\rho) M_0(N\cdot)+M_1(N\cdot) + A)^{-1}-((i\xi+\rho) M_0^{\textnormal{av}}+M_1^{\textnormal{av}} + A)^{-1}\right)(i\xi+\rho)^{-2}\|_{L(H))}\\
        & \leq \frac{\kappa}{N}(1+|\xi|^2)(1+\|M_0\|_\infty+\|M_1\|_\infty)^2\frac{1}{|i\xi +\rho|^2}=\frac{\kappa}{N}\frac{1+\xi^2}{\rho^2+\xi^2}(1+\|M_0\|_\infty+\|M_1\|_\infty)^2,
 \end{align*}
 which implies \eqref{eq:FLTest} and, thus, the assertion.
\end{proof}

  \section{Numerical implementation}\label{sec:numerics}
    We use as numerical method a discontinuous Galerkin method in time and a continuous
    Galerkin method in space. For a similar problem this approach is already considered and analysed in \cite{FrTW16}.
    Therefore, we will only describe the method here shortly and point to the differences
    in the numerical analysis.
    
    \subsection{Numerical method}
    We will start by describing the method and providing a convergence result for an arbitrary problem of type \eqref{eq:Lu}, that is, we shall focus on problems of the type
    \begin{equation}\label{eq:Lu1}
       \left(\partial_t M_0 + M_1 +A\right) U = F,
    \end{equation}
    where $A=\begin{pmatrix} 0 & \partial_\# \\ \partial_\# & 0       
             \end{pmatrix}$ and $M_0=M_0^\ast\geq 0$, $M_1$ are in $L^\infty(\Omega)^{2\times 2}$, which are readily extended to operators acting on $L_\rho^2(\mathbb{R};L^2(\Omega)^2)$. Throughout, we shall assume 
             \[
               \rho \langle M_0 \phi,\phi\rangle + \Re \langle M_1\phi,\phi\rangle\geq c\langle \phi,\phi\rangle
             \]for some $c>0$ and all $\phi\in L^2(\Omega)^2$.
    
    Let the time-interval $[0,T]$ be partitioned into subintervals $I_m=(t_{m-1},t_m]$ of 
    length $\tau_m$ for $m\in\{1,2,\ldots,M\}$ with $t_0=0$ and $t_M=T$. 
    Let the space-interval $\Omega:=(0,1)$ also be partitioned into subintervals $J_k=[x_{k-1},x_k]$
    of length $h_k$ for $k\in\{1,2,\dots,K\}$ with $x_0=0$ and $x_K=1$.
    Furthermore, let a temporal-polynomial degree $q\in \N$ and a spatial-polynomial degree $p\in\N$ be given. 
    
    Then we define the discrete space 
    \[
      \U^{h,\tau}
      \!:=\! \left\{
            (u_h,v_h)\in H_\rho(\R;H):
            u_h|_{I_m},
            v_h|_{I_m}\!\in\PS_q(I_m;V(\Omega)),
            m\!\in\!\{1,\dots,M\}
         \right\},
    \]
    where the spatial space is
    \begin{align*}
      V(\Omega)
        &:=\left\{v\in H^1_{\#}(\Omega);\,v|_{J_k}\in\PS_p(J_k),k\!\in\!\{1,\dots,K\}\right\},
    \end{align*}
    Furthermore, $\PS_q(I_m)$ is the space of polynomials of degree up to $q$ on the interval $I_m$ 
    and similarly $\PS_p(J_k)$. Thus our discrete space consists of function that are
    piece-wise polynomials of degree $p$ and continuous w.r.t.~the space variable, and
    piece-wise polynomial of degree $q$ and discontinuous at the time-points $t_k$ w.r.t.~time.
    
    The method reads: 
    For given $F\in \U^{h,\tau}$ and $x_0\in H$, find $\U\in\U^{h,\tau}$, such that for 
    all $\Phi\in \U^{h,\tau}$ and $m\in\{1,2,\dots,M\}$ it holds
    \begin{equation}\label{eq:discr_quad_form}
      \Qmr{(\partial_t M_0+M_1+A)\U,\Phi}
        +\scp{M_0 \jump{\U}_{m-1}^{x_0},{\Phi}^+_{m-1}}
        =\Qmr{ F,\Phi }.
    \end{equation}
      
    Here, we denote by 
    \[
      \jump{\U}_{m-1}^{x_0}:=
      \begin{cases}
        \U(t_{m-1}+)-U(t_{m-1}-),& m\in\{2,\ldots,M\}\\ 
        \U(t_0+)-x_0,& m=1,
      \end{cases} 
    \]
    the jump at $t_{m-1}$, by $\Phi^+_{m-1}:= \Phi(t_{m-1}+)$ and by 
    \[
      \Qmr{a,b}:
                =\frac{\tau_m}{2} \sum_{i=0}^q {\omega}^m_i \scp{a(\tmi),b(\tmi)}
    \]
    a right-sided weighted Gau\ss--Radau quadrature formula on $I_m$ approximating 
    \[
      \scprm{a,b} \coloneqq 
        \intop_{t_{m-1}}^{t_m} 
        \scp{a(t),b(t)} \exp(-2\rho(t-t_{m-1})) \mathrm{d} t,
    \]
    see \cite{FrTW16} for further details. We denote by $U^{h,\tau}_N$ the numerical solution obtained
    by above method \eqref{eq:discr_quad_form} for the problem with periodic, rough coefficients
    and by $U^{h,\tau}$ for the homogenised data.
    
    \subsection{Numerical analysis}
    We are ready to provide the convergence result for the above method assuming enough 
    regularity of the solution of Example \eqref{eq:Lu1} measuring the error in an $L^\infty$-$L^2$ 
    sense with
    \[
      E^2_{\sup}(a):=\sup_{t\in[0,T]}\langle M_0 a(t),a(t)\rangle_{L^2(\Omega)^2}
    \]
    and with the discrete version of the $L^2_\rho(\R;H)$-norm, given by
    \[
      E^2_Q(a):=\e^{2\rho T}\sum_{m=1}^M\Qmr{a,a}\e^{-2\rho t_{m-1}}.
    \]
    \begin{thm}\label{thm:conv_numer}
      We assume for the solution $U=(u,v)$ of Example \eqref{eq:Lu1} the 
      regularity     
      \[
        U\in H_\rho^{1}(\R;H_{\#}^p(\Omega)\times H_{\#}^p(\Omega))\cap 
             H_\rho^{q+3}(\R;L^2(\Omega)\times L^2(\Omega)) 
      \]
      as well as 
      \[
        AU\in H_\rho(\R; H_{\#}^p(\Omega)\times H_{\#}^p(\Omega)).
      \]
      Then we have for the error of the numerical solution by 
      \eqref{eq:discr_quad_form} with a generic constant $C$
      \[
        E^2_{\sup}(U-U^{h,\tau})+
        E^2_Q(U-U^{h,\tau})
        \leq C \e^{2\rho T}(\tau^{2(q+1)} + T h^{2p}).
      \]
    \end{thm}
    \begin{proof}
      The proof is basically identical to the one given in \cite{FrTW16}. The only difference
      being the periodic boundary condition instead of the homogeneous Dirichlet condition.
      But all estimates are the same, as only local estimates in space are used, independent of
      boundary conditions.
    \end{proof}
    
    Considering now the problem coming from the homogenisation process, we essentially have 
    two different problems we can approximate numerically, see Figure~\ref{fig:sketch}, where 
    in addition $U_N$ denotes the solution to the problem with rough coefficients.
    \begin{figure}[tb]
      \begin{center}
        \includegraphics[width=0.5\textwidth]{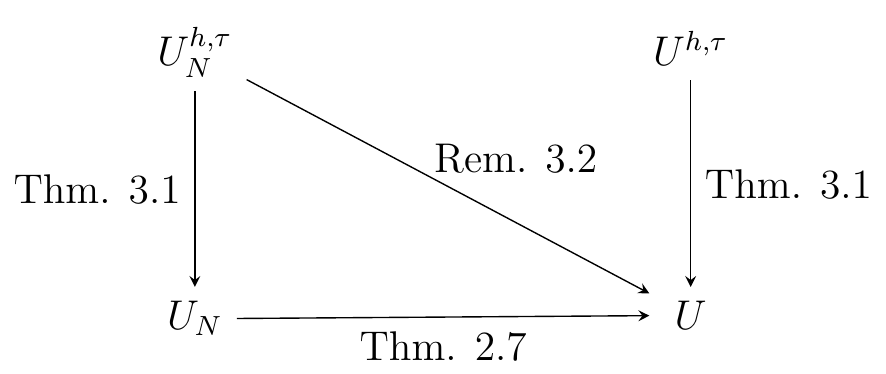}
      \end{center}
      \caption{Diagram showing the connections between the different problems\label{fig:sketch}}
    \end{figure}
  \begin{rem}\label{rem:conv}
    Following the diagram in Figure\ref{fig:sketch}, we have by the 
    Theorems~\ref{thm:conv_analy} and \ref{thm:conv_numer}
    for a the suitable choice of polynomial degrees
    $
      p=q+1\geq 1
    $
    and meshwidths
    $
      \tau=c_1h=\frac{c_2}{N},\,c_1,c_2>0
    $
    the convergence result
    \[
      E_Q(U_N^{h,\tau}-U)
        \leq E_Q(U_N^{h,\tau}-U_N)+E_Q(U_N-U)
        \leq E_Q(U_N^{h,\tau}-U_N)+C\|U_N-U\|_{H^1_\rho(\R,H)}        
        \leq C N^{-1},
    \]
    where the second inequality comes from Sobolev's embedding theorem (see e.g.~\cite[Lemma 5.2]{KPSTW14_OD})
    and the final one from applying Theorems~\ref{thm:conv_analy} and \ref{thm:conv_numer}. Note that for this estimate to hold we have to impose suitable regularity in time for the right-hand side in \eqref{eq:Lu} (or \eqref{eq:Lu1}).
  \end{rem}

  \subsection{Numerical example}\label{sec:simulation}    
    Let $N\in\N$ be even and with 
    \[
     \eps_N(x):=\begin{cases}
                  1,& \exists i\in\N_0:x\in\left[\frac{2i}{N},\frac{2i+1}{N}\right)\\
                  0,& otherwise
                \end{cases},\quad
     \sigma_N(x):=1-\eps_N(x)
    \]
    we consider the rough-coefficient problem for $U_N=(E_N,H_N)$
    \begin{gather}\label{eq:prob1}
      \left( 
        \partial_t
        \begin{pmatrix}
          \eps_N & 0\\
               0 & 1
        \end{pmatrix}
        +\begin{pmatrix}
          \sigma_N & 0\\
                 0 & 0
        \end{pmatrix}
        +\begin{pmatrix}
          0 & \partial_\# \\
          \partial_\# & 0
        \end{pmatrix}
      \right)\begin{pmatrix}
              E_N\\
              H_N
            \end{pmatrix}
      =\begin{pmatrix}
        J\\
        K
        \end{pmatrix}
    \end{gather}
    and the homogenised problem for $U=(E,H)$
    \begin{gather}\label{eq:prob1_hom}
      \left( 
        \partial_t
        \begin{pmatrix}
          \frac{1}{2} & 0\\
               0 & 1
        \end{pmatrix}
        +\begin{pmatrix}
          \frac{1}{2} & 0\\
                 0 & 0
        \end{pmatrix}
        +\begin{pmatrix}
          0 & \partial_\# \\
          \partial_\# & 0
        \end{pmatrix}
      \right)\begin{pmatrix}
              E\\
              H
            \end{pmatrix}
      =\begin{pmatrix}
        J\\
        K
        \end{pmatrix},
    \end{gather}
    where $J(t,x)=\sin(2\pi x)\cdot\min\{1,10t\}$ and $K(t,x)=0$ for all $t>0,x\in[0,1]$
    For our numerical experiment we use the Matlab/Octave software SOFE \cite{larscode}.
    The exact solutions are unknown. Therefore, we use reference solutions computed on a very 
    fine grid and higher polynomial degree in the computation of the errors.
    
    In Table~\ref{tab:prob1}
    \begin{table}[tb]
       \caption{Convergence results for $U_N-U_N^h$ and $U-U_N^h$ of problem \eqref{eq:prob1}
                \label{tab:prob1}}
       \begin{center}   
        \begin{tabular}{rllllllll}
           $n$ & 
           \multicolumn{2}{c}{$E_{\sup}(U_N-U_N^{h,\tau})$} &
           \multicolumn{2}{c}{$E_Q(U_N-U_N^{h,\tau})$} & 
           \multicolumn{2}{c}{$E_{\sup}(U-U_N^{h,\tau})$} &
           \multicolumn{2}{c}{$E_Q(U-U_N^{h,\tau})$}\\
           \hline
            4 & 2.857e-03 &      &  1.117e-03 &      & 1.381e-01 &      & 3.683e-02 &      \\
            8 & 9.490e-04 & 1.59 &  3.623e-04 & 1.62 & 3.418e-02 & 2.01 & 1.297e-02 & 1.51 \\
           16 & 2.802e-04 & 1.76 &  1.151e-04 & 1.65 & 1.328e-02 & 1.36 & 4.463e-03 & 1.54 \\
           32 & 8.611e-05 & 1.70 &  3.713e-05 & 1.63 & 5.890e-03 & 1.17 & 2.039e-03 & 1.13 \\
           64 & 2.306e-05 & 1.90 &  9.136e-06 & 2.02 & 2.802e-03 & 1.07 & 9.983e-04 & 1.03
        \end{tabular}
       \end{center}
    \end{table}
    we present the simulation results of $U_N^{h,\tau}=(E_N^{h,\tau},H_N^{h,\tau})$ for 
    $h=1/K$, $\tau=1/M$ and $M=2K=8N$ and polynomial degrees $p=q+1=2$. 
    In the second and third column we see almost second order convergence of 
    $U_N^{h,\tau}$ towards $U_N=(E_N,H_N)$ in accordance with Theorem~\ref{thm:conv_numer}, while in the last two columns
    we observe first order convergence of $U_N^{h,\tau}$ towards $U=(E,H)$ in accordance with Remark~\ref{rem:conv}.
    
    \section*{Acknowledgements}
    
    The authors wish to thank Shane Cooper for useful discussion on the subject and in particular on the discrete version of the Gelfand transformation presented here.

  \bibliographystyle{plain}
  
\end{document}